\documentclass[12pt, reqno]{amsart}
\usepackage{amsmath}
\usepackage{amsthm}
\usepackage{amssymb}
\usepackage{enumerate}
\usepackage{graphicx}
\usepackage{graphicx}
\usepackage{amssymb,amsfonts}
\usepackage{amsmath}
\usepackage[colorlinks,linkcolor=blue,anchorcolor=blue,citecolor=blue]{hyperref}
\usepackage[numbers,sort&compress]{natbib}
\usepackage{marvosym}
\usepackage{epstopdf}
\usepackage{caption}
\usepackage{multicol}
\usepackage{multirow}
\usepackage{makebox}
\usepackage{subfigure}
\numberwithin{equation}{section}
\newtheorem{theo}{Theorem} 

\newtheorem{lem}{Lemma}
\newtheorem{cor}{Corollary}
\newtheorem{remark}{Remark}

\begin{document}

\title{On some  discrete Bonnesen-style isoperimetric inequalities}

\author[Zeng and Dong]{Chunna Zeng and Xu Dong}

\address{School of Mathematics and Statistics,
 Chongqing Normal University,
Chongqing 401332, People's Republic of China}
\email{zengchn@163.com}

\address{School of Mathematics and Statistics,
 Chongqing Normal University,
Chongqing 401332, People's Republic of China}
\email{dx13370724790@163.com}


\thanks{This work is supported in part by the Major Special Project of NSFC (Grant No. 12141101), the Young Top-Talent program of Chongqing (Grant No. CQYC2021059145), NSF-CQCSTC (Grant No. cstc2020jcyj-msxmX0609) and Technology Research Foundation of Chongqing Educational committee (Grant No. KJQN201900530, KJZD-K202200509).}

\begin{abstract}
This article deals with the  sharp discrete  isoperimetric inequalities  in analysis and geometry for planar convex polygons. First, the  analytic isoperimetric inequalities based on Schur convex function  are established. In the wake of the analytic isoperimetric inequalities, Bonnesen-style isoperimetric inequalities and  inverse  Bonnesen-style inequalities for the planar convex polygons are obtained.


\end{abstract}

\maketitle

\section{Introduction }
The isoperimetric problem is an ancient problem in geometry that has  a significant impact on various branches of mathematics. The problem deals with finding a closed curve or surface that encloses the maximum area or volume for a given perimeter or boundary length. The geometric inequalities derived from  the isoperimetric problem have far-reaching effects and connections to other areas of mathematics.
In the 1950s, a connection was discovered between the isoperimetric problem and the Sobolev embedding problem. The link between the two problems provided insights into the relationship between geometric and functional analysis. In the 1970s, the Aleksandrov-Fenchel inequality, which is a generalization of the isoperimetric inequality, was found to be closely related to the Hodge index theorem in algebraic geometry. In recent decades, there has been extensive research on convex geometry related to the isoperimetric problem. This research has established important connections among various fields, including functional analysis, harmonic analysis, affine geometry, partial differential equations, and information theory.

 The complete mathematical proof of isoperimetric inequality was not established until the variational method based on calculus appeared in the 19th century. After this breakthrough, various proofs for the isoperimetric inequality were developed. In 1962, Hurwitz proposed two ingenious methods to prove the isoperimetric inequalities. He employed Fourier analysis in his proofs, one for convex curves and the other for general curves. These methods were detailed in his works \cite[Section 4.2]{Groemer H1996Geometric} and \cite[Page 392-394]{Hurwitz A1902Sur}, respectively.
 The planar isoperimetric inequality can be derived as an immediate consequence of  Poincar\'{e} formula \cite[Section 23]{Blaschke1949Chelsea}, which can be founded in references
  such as \cite[Chapter 7, Section 7]{Hardy G1988Inequalities} and \cite[Page 1183-1185]{Ossermam1978The}. Zhang demonstrated that the Sobolev inequality, a key result in functional analysis, is a special form of the isoperimetric inequality (\cite{Zhang G1999The affine Sobolev}.
  Zhou  utilized the fundamental  kinematic formula  to establish  Bonnesen-type isoperimetric inequalities and other new inequalities  in integral geometry (\cite{Ren D1994 Topics, Kinematic1, Kinematic2}).  It is worth mentioning that  Zhang \cite{Zhang X1996ARefine} introduced the discrete Wirtinger inequality to obtain some new analytical isoperimetric inequalities.

The description of classical isoperimetric inequality in the Euclidean plane $\mathbb R^{2}$ is: assume that $K$ is a domain with length $L$ and area $A$, then
\begin{equation*}
L^2-4\pi A\geq 0,
\end{equation*}
where the equality holds if and only if $K$ is a disc.

In 1920, Bonnesen discovered  a series of inequalities  in $\mathbb{R}^2$ as follows
(\cite{Ren D1994 Topics, Ossermam1978The, Ossermam1979, Zeng C2010The Bonnesen isoperimetric})
\begin{equation}
\Delta\left(K\right)=L^2-4\pi A\geq B_K,
\end{equation}
where $B_k$ is non-negative with geometric significance and vanishes only when $K$ is a disc.

Denote by $\Delta\left(K\right)=L^2-4\pi A$ the isoperimetric deficit of  $K,$ and $B_K$  measures the ``deviation"  between $K$ and a disc.
Many $B_K$'s for planar  domains  are found by  mathematicians. For instance, Bonnesen, Zhang, Zhou etc.~obtained a series of Bonnesen-type isoperimetric inequalities (\cite{Ossermam1978The, Ossermam1979, Zeng2012Thesymmetric}): assume that $K$ is the domain of length $L$ and area $A.$ Let $r_K$ and $R_K$ be the maximum inscribed radius  and minimum circumscribed radius, respectively. Then
\begin{equation*}
\pi r^2-Lr+A\le 0;
\end{equation*}
\begin{equation*}
\frac{L-\sqrt{L^2-4\pi A}}{2\pi}\le r_K\le\frac{L}{2\pi}\le R_k\le\frac{L-\sqrt{L^2+4\pi A}}{2\pi};
\end{equation*}
and
\begin{equation*}
L^2-4\pi A\ge A^2\left(\frac{1}{R_K}-\frac{1}{r_K}\right)^2;\ \ \ \ L^2-4\pi A\ge L^2\left(\frac{R_K-r_K}{R_K+r_K}\right)^2;
\end{equation*}

\begin{equation*}
L^2-4\pi A\ge A^2\left(\frac{1}{R_K}-\frac{1}{r}\right)^2;\ \ \ \ L^2-4\pi A\ge L^2\left(\frac{r-R_K}{r+R_K}\right)^2;
\end{equation*}

\begin{equation*}
L^2-4\pi A\ge A^2\left(\frac{1}{r}-\frac{1}{r_K}\right)^2;\ \ \ \ L^2-4\pi A\ge L^2\left(\frac{r_K-r}{r_K+r}\right)^2.
\end{equation*}
With equalities hold when and only when $K$ is a disc.

Comparing these $B_K$'s and determining the best lower bound poses a challenge, and mathematicians are still pursuing to discover  these unknown invariants of geometrical significance.

Rencently, mathematicians have  turned their  attention towards   the discrete isoperimetric problem, such as polygons or polyhedrons.
In fact, by establishing a series of  analytic inequalities that related to solutions of  nonlinear second-order differential equalities, Zhang \cite[Section 3]{Zhang X1996ARefine} obtained some discrete isoperimetric inequalities of polygons in $\mathbb{R}^2$.

Assume that $H_n$ is an $n$-sided planar convex  polygon. Denote by $L_n$ and $A_n$   length and  area of $H_n$, respectively. Then
\begin{equation}\label{fc12}
L_n^2-4d_nA_n\geq0,\ \ \ \ d_n=n\tan{\frac{\pi}{n}}.
\end{equation}
The quantity $L_n^2-4d_nA_n$ is called  the isoperimetric deficit of  $H_n$. The sharp discrete Bonnesen-style isoperimetric inequality  has the following form
\begin{equation}\label{fc13}
L_n^2-4d_nA_n\geq U_n,
\end{equation}
where $U_n$ is  non-negative with geometric significance and vanishes only when $H_n$ is a regular polygon.
$U_n$ measures the ``deviation" of $H_n$ from ``regularity".

Utilizing analytic inequalities such as the discrete Wirtinger inequality and Schur-convex (concave) function have been proved to be an  effective method  for solving discrete isoperimetric problem (see \cite{Wang W2018Schur, Zhang X1996ARefine, Zhang X1997Bonnesen-style, Zhang X1998Schur-Convex}). However, the discrete Bonnesen-style isoperimetric inequalities remain largely unexplored, with only a few  inequalities discovered in $\mathbb R^{2}$.
As noted in Zhang's paper  \cite{Zhang X1996ARefine}, a convex polygon that is inscribed in a circle  is called as ``cyclic" and such polygons always enclose the largest area. Therefore, in studying geometric inequalities for planar convex polygon, one only needs to consider cyclic polygons.

Assume that ${\Lambda}_n$ is an $n$-sided planar convex  polygon with perimeter  $L_n$ and area $A_n$, which is  inscribed in a circle with radius $R$. Meanwhile, there exist a regular polygon with perimeter $L_n^\ast$ and area $A_n^\ast$ inscribed in the same circle  ${\Lambda}_n$.   Zhang \cite{Zhang X1997Bonnesen-style} achieved the following geometric inequality for planar convex polygon in $\mathbb{R}^2$
\begin{equation}
L_n^2-4d_nA_n\geq\left(L_n^\ast-L_n\right)^2,\ \ \ \ d_n=n\tan{\frac{\pi}{n}},
\end{equation}
where equality holds  when and only when $\Lambda_n$ is a regular polygon.


In this paper, inspired by Zhang, Qi and Ma's work (\cite{Zhang X1998Schur-Convex,Wang W2018Schur}),  by establishing two analytic isoperimetric inequalities (Theorem \ref{qb}, \ref{2}), we derive a series of discrete Bonnesen-style isoperimetric inequalities (Theorem \ref{3}, \ref{4}) and their reverse forms (Theorem \ref{theo5}, \ref{6}). In Section five,
we  make up for the shortcomings in the proof of Ma's  theorem in  \cite{Ma L2015A Bonnesen-style},
and  present a new proof of  Zhang and Ma's results (\cite{Zhang X1997Bonnesen-style,Ma L2015A Bonnesen-style})(Theorem \ref{qe}). Meanwhile, we suppose that on surfaces of constant curvature, Schur convexity also will be a convenient and excellent idea to explore geometric  inequalities on surfaces of constant curvature.
%
%
%
%

\section{ isoperimetric inequalities in analysis}

The Schur-convex (concave)  function was introduced by  Schur in 1923  (\cite{Schur I1923Uber}) and  has many important applications in analytic inequalities. In this section,
we review some basic facts  about  Schur-convex (concave)  function and give two important  analytic inequalities.

An $n\times n$ matrix $P=\left[p_{ij}\right]$ is called a  doubly stochastic matrix if $p_{ij}\geq0$ for $1\le i$,~$j\le n$, and
\begin{equation*}
\sum_{j=1}^{n}p_{ij}=1,\ \ i=1, 2, \cdots, n;\ \ \ \ \  \sum_{\ i=1}^{n}p_{ij}=1,\ \ j=1, 2, \cdots, n.
\end{equation*}
For instance, $P=\left[p_{ij}\right]$ with $p_{ij}=\frac{1}{n}$, $1\le i$,~$j\le n$ is  a doubly stochastic matrix.

Let $I$ be an open interval of the real number line $\textbf{R}$, and  $I^n=I\times I\times\cdots\times I$.
A real function $f:I^n\rightarrow\textbf{R}$ ($n\geq 2$) is Schur-convex if for any doubly stochastic matrix $P$ and  $x\in I^n$, then
\begin{equation}\label{fc21}
f(Px)\le f(x).
\end{equation}
 Furthermore, $f$ is called  strictly Schur-convex if  (\ref{fc21}) is strict.~Similarly, $f$ is  strictly Schur-concave if  (\ref{fc21}) is inverse and strict.

A real function $f: I^n\rightarrow \textbf{R}$ ($n\geq 2$) is   symmetric if for any permutation matrix $T$ and  $x\in I^n$
\begin{equation}
f(Tx)=f(x).
\end{equation}

Every Schur-convex function is symmetric, but not every symmetric function can be a Schur-convex function. The following lemma is defined as the Schur's condition.
\begin{lem}\emph{(\cite{Wayne A1973Convex Functions})}\label{qz}
Assume that the real function $f\left(x\right)=f\left(x_1,x_2,\cdots,x_n\right)$ with symmetry has continuous partial derivatives on $I^n$. Then
$f: I^n\rightarrow \textbf{R}$ is Schur-convex (concave) when and only when
\begin{equation}\label{fc23}
\left(x_i-x_j\right)\left(\frac{\partial f}{\partial x_i}-\frac{\partial f}{\partial x_j}\right)\geq0\left(\le0\right).
\end{equation}
Furthermore $f$ is  strictly Schur-convex if inequality (\ref{fc23}) is strict for $x_i\neq x_j$,~$1\le i$,~$j\le n$.
\end{lem}

Due to the symmetry of  $f(x)$, the Schur's condition (\ref{fc23}) can be rewritten as (\cite[Page 57]{Albert1979Inequalities})
\begin{equation}\label{fc24}
\left(x_1-x_2\right)\left(\frac{\partial f}{\partial x_1}-\frac{\partial f}{\partial x_2}\right)\geq0\left(\le0\right),
\end{equation}
and $f$ is strictly Schur-convex if  (\ref{fc24}) is strict for $x_1\neq x_2$.

The following notations will be often used in the latter part of this article.

\begin{equation*}
I=\left(0,l\right);\ \ \ H_n=\Big\{\mathrm{\Theta}=\left(\theta_1,\cdots,\theta_n\right)\in \mathbb{R}^n,~\sum_{i=1}^{n}\theta_i=ml\Big\}\left(0<m<n\right);
\end{equation*}
\begin{equation*}
D_n=I^n\cap H_n;\ \ \ \mathrm{\Omega}=\left(\sigma,\cdots,\sigma\right)~where  ~\sigma=\frac{1}{n}\sum_{i=1}^{n}\theta_i=\frac{ml}{n}.
\end{equation*}

\begin{lem}\label{qccc}
If  $f$: $I^n\rightarrow \textbf{R}$ is a  Schur-concave function on $I^n$, then $f\left(\mathrm{\Omega}\right)$ is a global maximum in $D_n$. If  $f$ is strictly Schur-concave on $I^n$, then $f\left(\mathrm{\Omega}\right)$ is the unique global maximum in $D_n$.
\end{lem}
\begin{proof}
Because that $f$ is Schur-concave, then
\begin{equation*}
f\left(P\mathrm{\Theta}\right)\geq f\left(\mathrm{\Theta}\right),\ \ \ \Theta \in I^n
\end{equation*}
for any doubly stochastic matrix $P$. Then for $\mathrm{\Theta}=\left(\theta_1,\cdots,\theta_n\right)$ in $D_n$, and
\begin{equation*}
P=\left[p_{ij}\right]\ \ with\ \ p_{ij}=\frac{1}{n}, 1\le i,j\le n.
\end{equation*}
It yields
\begin{equation*}
f\left(\mathrm{\Omega}\right)=\ f\left(P\mathrm{\Theta}\right)\geq f\left(\mathrm{\Theta}\right).
\end{equation*}
If $f$ is strictly Schur-concave on $I^n$, then $f\left(\mathrm{\Omega}\right)>f\left(\mathrm{\Theta}\right)$ for all $\mathrm{\Theta}$ in $D_n$, $\mathrm{\Theta}\neq\mathrm{\Omega}$.
\end{proof}

\begin{lem}\emph{(\cite{Zhang X1998Schur-Convex})}\label{qd}
If  $f$: $I^n\rightarrow \textbf{R}$ is a  Schur-convex function, then $f\left(\mathrm{\Omega}\right)$ is a global minimum in $D_n$. If  $f$ is  strictly Schur-convex  on $I^n$, then $f\left(\mathrm{\Omega}\right)$ is  unique global minimum in $D_n$.
\end{lem}

We obtain the following  analytic inequality.
\begin{theo}\label{qb}
Let  $f\left(\theta\right)$ be a positive, strictly convex function, then
\begin{equation}
\left(\sum_{i=1}^{n}f\left(\theta_i\right)\right)^{2\alpha}-\left(nf\left(\sigma\right)\right)^\alpha\left(\sum_{i=1}^{n}f\left(\theta_i\right)\right)^\alpha\geq\left(f\left(\sigma\right)\right)^\alpha\left[\left(\sum_{i=1}^{n}f\left(\theta_i\right)\right)^\alpha-\left(nf\left(\sigma\right)\right)^\alpha\right],
\end{equation}
where  $\alpha\in \mathbb N^{+}.$ The equality holds when and only when $\theta_1=\theta_2=\cdots=\theta_n=\sigma$.
\end{theo}
\begin{proof}
Consider the following function
\begin{equation}
F\left(\Theta\right)=\left(\sum_{i=1}^{n}f\left(\theta_i\right)\right)^{2\alpha}-\left(nf\left(\sigma\right)\right)^\alpha\left(\sum_{i=1}^{n}f\left(\theta_i\right)\right)^\alpha
-\left(f\left(\sigma\right)\right)^\alpha\left[\left(\sum_{i=1}^{n}f\left(\theta_i\right)\right)^\alpha-\left(nf\left(\sigma\right)\right)^\alpha\right].
\end{equation}
It shows that $F\left(\Theta\right)$ is  symmetric on $I^n$ and $F\left(\mathrm{\Omega}\right)=0$. In order to prove Theorem \ref{qb}, we need to prove that
$F\left(\Theta\right)\geq 0$ and $F\left(\Theta\right)$  is strictly Schur-convex on $I^n$. Then by Lemma \ref{qd} it implies that $F\left(\mathrm{\Theta}\right)$ has the unique global minimum $F\left(\mathrm{\Omega}\right)$ in $D_n$. So the key of this proof is to determine whether   $F\left(\Theta\right)$ satisfies  strict Schur-convexity
on $I^n$. By Lemma \ref{qz} and (\ref{fc24}), in order to demonstrate that, it is necessary to  verify the following inequality
\begin{equation}\label{fc27}
\left(\theta_1-\theta_2\right)\left(\frac{\partial F}{\partial\theta_1}-\frac{\partial F}{\partial\theta_2}\right)>0,\ \ if\ {\ \theta}_1\neq\theta_2.
\end{equation}
Let $P_n=\sum_{i=1}^{n}f\left(\theta_i\right)$, then
\begin{equation*}
F\left(\Theta\right)=\left(P_n\right)^{2\alpha}-\left(n^\alpha+1\right)\left(f\left(\sigma\right)\right)^\alpha\left(P_n\right)^\alpha+{n^\alpha\left(f\left(\sigma\right)\right)}^{2\alpha}.
\end{equation*}
Therefore
\begin{equation*}
\frac{\partial F}{\partial\theta_i}=2\alpha\left(P_n\right)^{2\alpha-1}f^\prime\left(\theta_i\right)-\alpha\left(n^\alpha+1\right)\left(f\left(\sigma\right)\right)^\alpha\left(P_n\right)^{\alpha-1}f^\prime
\left(\theta_i\right)
\end{equation*}
and
\begin{equation*}
\begin{aligned}
\left(\theta_1-\theta_2\right)\left(\frac{\partial F}{\partial\theta_1}-\frac{\partial F}{\partial\theta_2}\right) =&\left(\theta_1-\theta_2\right)\cdot\left(f^\prime\left(\theta_1\right)-f^\prime\left(\theta_2\right)\right)
\\&\cdot\left[2\alpha\left(P_n\right)^{2\alpha-1}-\alpha\left(n^\alpha+1\right)\left(f\left(\sigma\right)\right)^\alpha\left(P_n\right)^{\alpha-1}\right].
\end{aligned}
\end{equation*}
Since $f$ has strict convexity, then $f^{\prime\prime}>0$ and $f^{\prime}$ is  increasing on $\left(0, l\right).$ This leads to
\begin{equation}\label{fc28}
\left(\theta_1-\theta_2\right)\left(f^\prime\left(\theta_1\right)-f^\prime\left(\theta_2\right)\right)>0.
\end{equation}
Consider
\begin{equation*}
G\left(\mathrm{\Theta}\right)=2\alpha\left(P_n\right)^{2\alpha-1}-\alpha\left(n^\alpha+1\right)\left(f\left(\sigma\right)\right)^\alpha\left(P_n\right)^{\alpha-1}.
\end{equation*}
Because that $f$ is strictly convex and by  Jensen's inequality it follows that
\begin{equation}\label{fc29}
\left(P_n\right)^\alpha>\left(nf\left(\sigma\right)\right)^\alpha.
\end{equation}

\noindent (\ref{fc29}) is strict and  it implies that
\begin{equation*}
\left(P_n\right)^\alpha+\left(P_n\right)^\alpha>\left(nf\left(\sigma\right)\right)^\alpha+\left(f\left(\sigma\right)\right)^\alpha=\left(n^\alpha+1\right)\left(f\left(\sigma\right)\right)^\alpha,
\end{equation*}
\noindent that is
\begin{equation}\label{fc210}
{2\left(P_n\right)}^\alpha>\left(n^\alpha+1\right)\left(f\left(\sigma\right)\right)^\alpha.
\end{equation}

\noindent  (\ref{fc210}) is equivalent to
\begin{equation*}
2\alpha\left(P_n\right)^{2\alpha-1}>\alpha\left(n^\alpha+1\right)\left(f\left(\sigma\right)\right)^\alpha\left(P_n\right)^{\alpha-1}.
\end{equation*}

\noindent Therefore
\begin{equation}\label{fc211}
G\left(\mathrm{\Theta}\right)=2\alpha\left(P_n\right)^{2\alpha-1}-\alpha\left(n^\alpha+1\right)\left(f\left(\sigma\right)\right)^\alpha\left(P_n\right)^{\alpha-1}>0.
\end{equation}

\noindent From (\ref{fc28}) and (\ref{fc211}), (\ref{fc27}) is proved.
So we complete the proof of Theorem  \ref{qb}.
\end{proof}

\begin{remark}
\emph{In the  proof of Theorem \ref{qb}, assume that $f\left(\theta\right)$ is  a positive convex (resp. concave) function, and taking $\alpha=1,$ then (\ref{fc29})  becomes the classical Jensen's inequality: }
\begin{equation}
\sum_{i=1}^{n}f\left(\theta_i\right)\geq \left(\le\right)nf\left(\sigma\right),
\end{equation}
\emph{with equalities   when and only when $\theta_1=\theta_2=\cdots=\theta_n=\sigma$. For example}

\begin{equation*}
 \left(1\right)\ \ \sum_{i=1}^{n}{\tan\theta_i}\geq n\tan{\sigma},\ \ \ \ \ \ \ \ \ \ \ \
\left(2\right) \ \ \sum_{i=1}^{n}{\sin\theta_i}\le n\sin{\sigma},
\end{equation*}
\emph{where $\theta_i\in\left(0,\frac{\pi}{2}\right)$, $i=1,2,\cdots,n$; and~ ~$\sum_{i=1}^{n}\theta_i=\pi$, ~$\sigma=\frac{1}{n}\sum_{i=1}^{n}\theta_i=\frac{\pi}{n}$.}
\end{remark}

In  particular, if $f\left(\theta\right)=\tan{\theta}$ and $f\left(\theta\right)=\sec{\theta}$ for $\theta\in\left(0,\frac{\pi}{2}\right)$, then

\begin{cor}
Let~$\theta_i\in\left(0,\frac{\pi}{2}\right)$, $i=1,2,\cdots,n$; and~ ~$\sum_{i=1}^{n}\theta_i=\pi$, then for $\alpha\in \mathbb N^{+}$
\begin{equation}\label{fc212}
\left(\sum_{i=1}^{n}\tan{\theta_{i}}\right)^{2\alpha}-\left(n\tan{\frac{\pi}{n}}\right)^\alpha\left(\sum_{i=1}^{n}\tan{\theta_{i}}\right)^\alpha\geq\left(\tan{\frac{\pi}{n}}\right)^\alpha\left[\left(\sum_{i=1}^{n}\tan{\theta_{i}}\right)^\alpha-\left(n\tan{\frac{\pi}{n}}\right)^\alpha\right];
\end{equation}
 Specially, if $\alpha=1$, then
\begin{equation}
\left(\sum_{i=1}^{n}\tan{\theta_{i}}\right)^2-\left(n\tan{\frac{\pi}{n}}\right)\left(\sum_{i=1}^{n}\tan{\theta_{i}}\right)\geq \tan{\frac{\pi}{n}}\left(\sum_{i=1}^{n}\tan{\theta_{i}}-n\tan{\frac{\pi}{n}}\right),
\end{equation}
both equalities hold  when and only when  $\theta_1=\theta_2=\cdots=\theta_n=\frac{\pi}{n}$.
\end{cor}

\begin{cor}
Let~$\theta_i\in\left(0,\frac{\pi}{2}\right)$, $i=1,2,\cdots,n$; and ~$\sum_{i=1}^{n}\theta_i=\pi$, then
\begin{equation}
\left(\sum_{i=1}^{n}\sec{\theta_{i}}\right)^{2\alpha}-\left(n\sec{\frac{\pi}{n}}\right)^\alpha\left(\sum_{i=1}^{n}\sec{\theta_{i}}\right)^\alpha\geq\left(\sec{\frac{\pi}{n}}\right)^\alpha\left[\left(\sum_{i=1}^{n}\sec{\theta_{i}}\right)^\alpha-\left(n\sec{\frac{\pi}{n}}\right)^\alpha\right],
\end{equation}
where  $\alpha\in \mathbb N^{+}.$ The equality holds when and only when  $\theta_1=\theta_2=\cdots=\theta_n=\frac{\pi}{n}$.
\end{cor}

Notice that $f\left(x\right)=x^2$ on $\left(0,1\right)$ has strict convexity, then
\begin{cor}
 Let $\sum_{i=1}^{n}x_i=m$, $i=1,2,\cdots,n$, then
\begin{equation}
\left(\sum_{i=1}^{n}x_i^2\right)^{2\alpha}-\left(\frac{m^2}{n}\right)^\alpha\left(\sum_{i=1}^{n}x_i^2\right)^\alpha\geq\left(\frac{m^2}{n^2}\right)^\alpha\left[\left(\sum_{i=1}^{n}x_i^2\right)^\alpha-\left(\frac{m^2}{n}\right)^\alpha\right].
\end{equation}

\noindent for $\alpha\in \mathbb N^{+}.$ The equality holds when and only when $x_1=x_2=\cdots=x_n=\frac{m}{n}$ $\left(0<m<n\right)$.
\end{cor}

\begin{theo}\label{2}
Let $f\left(\theta\right)$ be a positive, strictly convex function, then
\begin{equation}
\left(\sum_{i=1}^{n}f\left(\theta_i\right)\right)^{2\alpha}-\left(nf\left(\sigma\right)\right)^\alpha\left(\sum_{i=1}^{n}f\left(\theta_i\right)\right)^\alpha\le\left(\sum_{i=1}^{n}f\left(\theta_i\right)\right)^{k\alpha}-\left(nf\left(\sigma\right)\right)^{k\alpha},
\end{equation}
where $\alpha, k\in\mathbb N^{+}$ and $k\geq2$. The equality holds when and only when $\theta_1=\theta_2=\cdots=\theta_n=\sigma$.
\end{theo}
\begin{proof}

Consider the following function
\begin{equation}\label{fc217}
F\left(\Theta\right)=\left(\sum_{i=1}^{n}f\left(\theta_i\right)\right)^{2\alpha}-\left(nf\left(\sigma\right)\right)^\alpha\left(\sum_{i=1}^{n}f\left(\theta_i\right)\right)^\alpha-\left.\left(\sum_{i=1}^{n}f\left(\theta_i\right)\right)^{k\alpha}+\left(nf\left(\sigma\right)\right)^{k\alpha}.\right.
\end{equation}
It is obviously that $F\left(\Theta\right)$ is  symmetrical  on $I^n$ and $F\left(\mathrm{\Omega}\right)=0$. Lemma \ref{qccc} implies that $F\left(\mathrm{\Omega}\right)$ is the unique global maximum in $D_n$ when  $F\left(\Theta\right)$ is  strictly Schur-concave. So the key is to determine whether $F\left(\Theta\right)$ is  strictly Schur-concave  on $I^n$. By Lemma 1 and (\ref{fc24}), in order to demonstrate that, it is necessary to verify the following inequality
\begin{equation}\label{fc218}
\left(\theta_1-\theta_2\right)\left(\frac{\partial F}{\partial\theta_1}-\frac{\partial F}{\partial\theta_2}\right)<0,\ \ if\ {\ \theta}_1\neq\theta_2.
\end{equation}
Let $P_n=\sum_{i=1}^{n}f\left(\theta_i\right)$, then
\begin{equation*}
F\left(\Theta\right)=\left(P_n\right)^{2\alpha}-\left(nf\left(\sigma\right)\right)^\alpha\left(P_n\right)^\alpha-\left(P_n\right)^{k\alpha}+\left(nf\left(\sigma\right)\right)^{k\alpha}.
\end{equation*}
Therefore

\begin{equation*}
\begin{aligned}
\frac{\partial F}{\partial\theta_i}&=2\alpha\left(P_n\right)^{2\alpha-1}f^\prime\left(\theta_i\right)-n\left(f\left(\sigma\right)\right)^a\alpha\left(P_n\right)^{\alpha-1}f^\prime\left(\theta_i\right)-k\alpha\left(P_n\right)^{k\alpha-1}f^\prime\left(\theta_i\right)
\\&=f^\prime\left(\theta_i\right)\left[2\alpha\left(P_n\right)^{2\alpha-1}-n\left(f\left(\sigma\right)\right)^a\alpha\left(P_n\right)^{\alpha-1}-k\alpha\left(P_n\right)^{k\alpha-1}\right]
\end{aligned}
\end{equation*}
and
\begin{equation}\label{fc219}
\begin{aligned}
\left(\theta_1-\theta_2\right)\left(\frac{\partial f}{\partial\theta_1}-\frac{\partial f}{\partial\theta_2}\right)=&\left(\theta_1-\theta_2\right)\cdot\left(f^\prime\left(\theta_1\right)-f^\prime\left(\theta_2\right)\right)
\\ &\cdot\left[2\alpha\left(P_n\right)^{2\alpha-1}-n\left(f\left(\sigma\right)\right)^a\alpha\left(P_n\right)^{\alpha-1}-k\alpha\left(P_n\right)^{k\alpha-1}\right].
\end{aligned}
\end{equation}
Since $f$ is strictly convex, then $f^{\prime\prime}>0$ and $f^{'}$ is increasing. It leads to
\begin{equation*}
\left(\theta_1-\theta_2\right)\left(f^\prime\left(\theta_1\right)-f^\prime\left(\theta_2\right)\right)>0,
\end{equation*}
and 
\begin{equation}\label{fc220}
2\alpha\left(P_n\right)^{2\alpha-1}-n\left(f\left(\sigma\right)\right)^a\alpha\left(P_n\right)^{\alpha-1}-k\alpha\left(P_n\right)^{k\alpha-1}<0.
\end{equation}
From \ (\ref{fc219}) and (\ref{fc220}), (\ref{fc218}) is obtained. So we complete the proof of Theorem \ref{2}.
\end{proof}

In particular, assume that
$f\left(\theta\right)=\tan{\theta}$ and $f\left(\theta\right)=\csc{\theta}$ for $\theta\in\left(0,\frac{\pi}{2}\right)$, then we have
\begin{cor}
Let~$\theta_i\in\left(0,\frac{\pi}{2}\right)$, $i=1,2,\cdots,n;$ and $\sum_{i=1}^{n}\theta_i=\pi$, then
\begin{equation}\label{fc221}
\left(\sum_{i=1}^{n}\tan{\theta_{i}}\right)^{2\alpha}-\left(n\tan{\frac{\pi}{n}}\right)^\alpha\left(\sum_{i=1}^{n}\tan{\theta_{i}}\right)^\alpha\le\left(\sum_{i=1}^{n}\tan{\theta_{i}}\right)^{k\alpha}-\left(n\tan{\frac{\pi}{n}}\right)^{k\alpha},
\end{equation}
where  $\alpha, k\in \mathbb N^{+}$ and  $k\ge2$. Especially, if $\alpha=1$ and $k=3$, then
\begin{equation}
\left(\sum_{i=1}^{n}\tan{\theta_{i}}\right)^2-\left(n\tan{\frac{\pi}{n}}\right)\left(\sum_{i=1}^{n}\tan{\theta_{i}}\right)\le\left.\left(\sum_{i=1}^{n}\tan{\theta_{i}}\right)^3-\left(n\tan{\frac{\pi}{n}}\right)^3,\right.
\end{equation}
both equalities hold when and only when  $\theta_1=\theta_2=\cdots=\theta_n=\frac{\pi}{n}$.
\end{cor}

\begin{cor}
Let $f\left(\theta\right)=\csc{\theta}$,~$\theta_i\in\left(0,\frac{\pi}{2}\right)$, $i=1,2,\cdots,n$; and $\sum_{i=1}^{n}\theta_i=\pi$, then
\begin{equation}
\left(\sum_{i=1}^{n}\csc{\theta_{i}}\right)^{2\alpha}-\left(n\csc{\frac{\pi}{n}}\right)^\alpha\left(\sum_{i=1}^{n}\csc{\theta_{i }}\right)^\alpha\le\left(\sum_{i=1}^{n}\csc{\theta_{i}}\right)^{k\alpha}-\left(n\csc{\frac{\pi}{n}}\right)^{k\alpha},
\end{equation}
where $\alpha, k\in \mathbb N^{+}$ and $k\ge 2$. Especially, when $\alpha=1$ and $k=3$, then
\begin{equation}
\left(\sum_{i=1}^{n}\csc{\theta_{i}}\right)^{2}-\left(n\csc{\frac{\pi}{n}}\right)\left(\sum_{i=1}^{n}\csc{\theta_{i }}\right)\le\left(\sum_{i=1}^{n}\csc{\theta_{i}}\right)^{3}-\left(n\csc{\frac{\pi}{n}}\right)^{3},
\end{equation}
both equalities hold when and only when  $\theta_1=\theta_2=\cdots=\theta_n=\frac{\pi}{n}$.
\end{cor}

\section{Discrete Bonnesen-style  isoperimetric inequalities}
In  this section, by using the  analytic isoperimetric inequalities in Theorem \ref{qb}, we establish some  discrete Bonnesen isoperimetric inequalities.

\begin{theo}\label{3}
Assume that ${\Gamma}_n$ is an   $n$-sided planar  convex polygon circumscribed in a circle of radius $R$. Denote by  $L_n$ and  $A_n$  the perimeter and area  of
${\Gamma}_n$, then
\begin{equation}\label{fc31}
L_n^{2\alpha}-4^\alpha\left(d_nA_n\right)^\alpha\geq 2^\alpha R^\alpha\left(\tan{\frac{\pi}{n}}\right)^\alpha\left[L_n^\alpha-\left(L_n^\ast\right)^\alpha\right],\ \ \ d_n=ntan{\frac{\pi}{n}},
\end{equation}
\begin{equation}\label{fc32}
\left(\frac{A_n}{R^2}\right)^{2\alpha}-d_n^\alpha\left(\frac{L_n}{2R}\right)^\alpha\geq\left(\tan{\frac{\pi}{n}}\right)^\alpha\left[\left(\frac{L_n}{2R}\right)^\alpha-\left(\frac{L_n^\ast}{2R}\right)^\alpha\right],\ \ \ d_n=ntan{\frac{\pi}{n}},
\end{equation}
where $\alpha\in\mathbb N^{+}$, $L_n^\ast$ is the perimeter of the regular  $n$-sided convex polygon circumscribed in the same circle with ${\Gamma}_n$.~Both equalities hold when and only when ${\Gamma}_n$ is a regular polygon.
\end{theo}
\begin{proof}
Denote by $a_i$ and $\theta_i$ the length of the $i$th side of ${\Gamma}_n$, and  the half of the central angle subtended by the $i$th vertex $A_i$ of ${\Gamma}_n$, $i=1,2,\cdots,n$, respectively. Then
\begin{figure}[h]
	\centering
	\includegraphics[width=5.6cm]{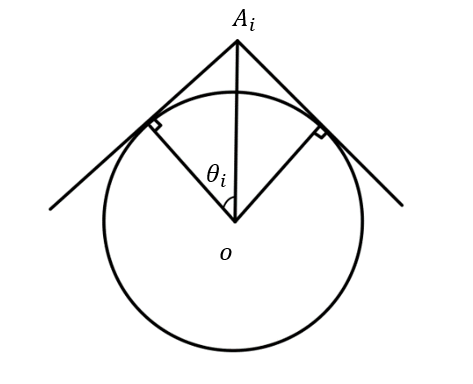}
	\label{1.png}
\end{figure}
\begin{equation}\label{fc33}
L_n=\sum_{i=1}^{n}a_i=2R\sum_{i=1}^{n}{\tan\theta_i};\ \ \ \ \ \ \ \ \ \ L_n^\ast=2nR\tan{\frac{\pi}{n}}.
\end{equation}
\begin{equation}\label{fc34}
\sum_{i=1}^{n}\theta_i=\pi;\ \ \ \ \ \ \ \ A_n=\frac{1}{2}\sum_{i=1}^{n}a_iR=R^2\sum_{i=1}^{n}{\tan\theta_i};{\ \ \ \ \ \ \ A}_n^\ast=nR^2\tan{\frac{\pi}{n}}.
\end{equation}
Substituting (\ref{fc33}) and (\ref{fc34}) into (\ref{fc212}) , we obtain (\ref{fc31}) and (\ref{fc32}).
\end{proof}
\begin{remark}
\emph{The inequality (\ref{fc32}) can be regard as the reverse inequality of (\ref{fc31}). }
\end{remark}

Assume that $\alpha=1$,  we have the following special case.
\begin{cor}
Assume that ${\Gamma}_n$ is an   $n$-sided  planar convex polygon circumscribed in a circle of radius $R$. Denote by  $L_n$ and  $A_n$  the perimeter and area  of
${\Gamma}_n$, then
\begin{equation}\label{fc35}
L_n^2-4d_nA_n\geq2R\left(\tan{\frac{\pi}{n}}\right)\left[L_n-L_n^\ast\right],
\end{equation}
\begin{equation}\label{fc36}
\left(\frac{A_n}{R^2}\right)^2-d_n\left(\frac{L_n}{2R}\right)\geq\left(\tan{\frac{\pi}{n}}\right)\left(\frac{L_n}{2R}-\frac{L_n^\ast}{2R}\right),
\end{equation}
where $d_n=n\tan{\frac{\pi}{n}}$. Both equalities hold when and only when ${\Gamma}_n$ is a regular polygon.
\end{cor}

Similarly, substituting (\ref{fc33}) and (\ref{fc34}) into (\ref{fc212}), we obtain
\begin{theo}\label{4}
Assume that ${\Gamma}_n$ is an   $n$-sided  planar convex polygon circumscribed in a circle of radius $R$. Denote by  $L_n$ and  $A_n$  the perimeter and area  of
${\Gamma}_n$, then
\begin{equation}\label{fc37}
L_n^{2\alpha}-4^\alpha\left(d_nA_n\right)^\alpha\geq 4^\alpha\left(\tan{\frac{\pi}{n}}\right)^\alpha\left[A_n^\alpha-\left(A_n^\ast\right)^\alpha\right],\ \ \ \ d_n=n\tan{\frac{\pi}{n}},
\end{equation}
\begin{equation}\label{fc38}
\left(\frac{A_n}{R^2}\right)^{2\alpha}-d_n^\alpha\left(\frac{L_n}{2R}\right)^\alpha\geq\left(\tan{\frac{\pi}{n}}\right)^\alpha\left[\left(\frac{A_n}{R^2}\right)^\alpha-\left(\frac{A_n^\ast}{R^2}\right)^\alpha\right],\ \ \ \ d_n=n\tan{\frac{\pi}{n}},
\end{equation}
where  $\alpha\in\mathbb N^{+}$, and $A_n^\ast$ denotes the area of the regular $n$-sided convex polygon  circumscribed in the same circle with ${\Gamma}_n$. Both equalities hold when and only when ${\Gamma}_n$ is a regular polygon.
\end{theo}
(\ref{fc38}) can be regard as the reverse inequality of (\ref{fc37}).
When $\alpha=1$, we have

\begin{cor}
Assume that ${\Gamma}_n$ is an   $n$-sided  planar convex polygon circumscribed in a circle of radius $R$. Denote by  $L_n$ and  $A_n$  the perimeter and area  of
${\Gamma}_n$, then
\begin{equation}\label{qx}
L_n^2-4d_nA_n\geq4\left(\tan{\frac{\pi}{n}}\right)\left[A_n-A_n^\ast\right],
\end{equation}
\begin{equation}\label{qc}
\left(\frac{A_n}{R^2}\right)^2-d_n\left(\frac{L_n}{2R}\right)\geq\left(\tan{\frac{\pi}{n}}\right)\left[\left(\frac{A_n}{R^2}\right)-\left(\frac{A_n^\ast}{R^2}\right)\right],
\end{equation}
where $d_n=n\tan{\frac{\pi}{n}}$. Both equalities hold when and only when ${\Gamma}_n$ is a regular polygon.
\end{cor}

\section{Inverse Discrete Bonnesen-type isoperimetric inequalities}
 It is interesting and also difficult to study the reverse  Bonnesen-style isoperimetric inequalities. There are only a few reverse  Bonnesen-style isoperimetric inequalities of convex domains (see \cite{Gao X2012A new reverse}). In 1993, Bottema discovered  a classic  reverse  Bonnesen-style isoperimetric inequality  of oval domains in $\mathbb{R}^2$(\cite{Ren D1994 Topics}). The discrete cases are more complex and difficult.

Assume that $K$ is a convex domain  with   $C^2$ smooth boundary in $\mathbb R^{2}$, then (\cite{Ren D1994 Topics})
\begin{equation*}
\mathrm{\Delta}\left(K\right)=L^2-4\pi A\le\pi^2\left(\rho_M-\rho_m\right)^2,
\end{equation*}
where $\rho_M$ and $\rho_m$ are the maximum  and minimum value of $\rho\left(\partial K\right)$, respectively.  The equality holds when  and only when $\partial K$ is a circle.

Mathematicians still continue to work hard on  investigating the inverse discrete Bonnesen-type isoperimetric inequalities.  A natural question is that: for an $n$-sided  planar convex polygon if there is a geometric invariant $D_n$ such that
\begin{equation*}
L_n^2-4d_nA_n\le D_n,\ \ \ d_n=n\tan{\frac{\pi}{n}},
\end{equation*}
where $D_n$ satisfies
\begin{enumerate}
  \item $D_n\geq 0$;
  \item $D_n=0$ only when $H_n$ is regular.
\end{enumerate}

In order to obtain the inverse discrete Bonnesen-style isoperimetric inequalities, we use the analysis isoperimetric  inequalities in  Theroem \ref{2}. Substituting (\ref{fc33}) and (\ref{fc34}) into (\ref{fc221}), then we have

\begin{theo}\label{theo5}
Assume that ${\Gamma}_n$ is an  $n$-sided  planar convex polygon circumscribed in a circle of radius $R$. Denote by  $L_n$ and  $A_n$  the perimeter and area  of
${\Gamma}_n$, then
\begin{equation}\label{fc41}
L_n^{2\alpha}-\left(4d_nA_n\right)^\alpha\le L_n^{k\alpha}-\left(L_n^\ast\right)^{k\alpha},\ \ \ d_n=n\tan{\frac{\pi}{n}},
\end{equation}
\begin{equation}\label{fc42}
\left(\frac{A_n}{R^2}\right)^{2\alpha}-d_n^\alpha\left(\frac{L_n}{2R}\right)^\alpha\le\left(\frac{L_n}{2R}\right)^{k\alpha}-\left(\frac{L_n^\ast}{2R}\right)^{k\alpha},\ 
\end{equation}
where  $\alpha, k\in\mathbb N^{+}$ and $k\geq 2$. Both equalities hold when and only when ${\Gamma}_n$ is a  regular polygon.
\end{theo}

Specially, we can regard inequality (\ref{fc42})  as the reverse inequality of (\ref{fc41}).

In particular, when $\alpha=1$ and $k=2,3$,   Theorem \ref{theo5} can be expressed as follows.
\begin{cor}
Assume that ${\Gamma}_n$ is an  $n$-sided  planar convex polygon circumscribed in a circle of radius $R$. Denote by  $L_n$ and  $A_n$  the perimeter and area  of
${\Gamma}_n$, then
\begin{equation}
L_n^2-4 d_nA_n\le L_n^3-(L_n^\ast)^3,
\end{equation}
\begin{equation}
\left(\frac{A_n}{R^2}\right)^2-d_n\left(\frac{L_n}{2R}\right)\le\left(\frac{L_n}{2R}\right)^2-\left(\frac{L_n^\ast}{2R}\right)^2,
\end{equation}
where $d_n=n\tan{\frac{\pi}{n}}$. Both equalities hold when and only when ${\Gamma}_n$ is a regular polygon.
\end{cor}

Substituting (\ref{fc33}) and (\ref{fc34}) into (\ref{fc221}), then we have
\begin{theo}\label{6}
Assume that ${\Gamma}_n$ is an  $n$-sided  planar convex polygon circumscribed in a circle of radius $R$. Denote by  $L_n$ and  $A_n$  the perimeter and area  of
${\Gamma}_n$, then
\begin{equation}\label{fc45}
L_n^{2\alpha}-4^\alpha\left(d_nA_n\right)^\alpha\le\frac{4^\alpha}{R^{2\left(k-1\right)\alpha}}\left[A_n^{k\alpha}-\left(A_n^\ast\right)^{k\alpha}\right],\ \ \ \ d_n=n\tan{\frac{\pi}{n}},
\end{equation}
\begin{equation}\label{fc46}
\left(\frac{A_n}{R^2}\right)^{2\alpha}-d_n^\alpha\left(\frac{L_n}{2R}\right)^\alpha\le\left(\frac{A_n}{R^2}\right)^{k\alpha}-\left(\frac{A_n^\ast}{R^2}\right)^{k\alpha},\ 
\end{equation}
where $\alpha, k$ and $k\ge 2$. Both equalities hold when and only when ${\Gamma}_n$ is a regular polygon.

\end{theo}

In particular, when $\alpha=1$ and $k=2, 3$,  Theorem \ref{theo5} also can  be expressed as follows.

\begin{cor}
Assume that ${\Gamma}_n$ is an  $n$-sided  planar convex polygon circumscribed in a circle of radius $R$. Denote by  $L_n$ and  $A_n$  the perimeter and area  of
${\Gamma}_n$, then
\begin{equation}
L_n^2-4d_nA_n\le\frac{4}{R^2}\left[A_n^2-\left(A_n^\ast\right)^2\right],
\end{equation}
\begin{equation}
\left(\frac{A_n}{R^2}\right)^2-d_n\left(\frac{L_n}{2R}\right)\le\left(\frac{A_n}{R^2}\right)^3-\left(\frac{A_n^\ast}{R^2}\right)^3,
\end{equation}
where $d_n=n\tan{\frac{\pi}{n}}$. Both equalities hold when and only when ${\Gamma}_n$ is a regular polygon.
\end{cor}

\section{some notes}
Many   special analytic inequalities have interesting geometric consequences.
In this section, we first prove the following analytic inequalities (Theorem \ref{qe}) in which we set two different
variables $\theta_i$ and $\psi_i$. As the consequence of Theorem  \ref{qe}, we present  new proofs of Zhang's result \cite{Zhang X1997Bonnesen-style} and Ma's result  (\cite{Ma L2015A Bonnesen-style}).
Note that there is a defect in the proof of Ma's main theorem. For completeness,
 we give the new proofs of these isoperimetric style inequalities.

Assume that $\theta_i$ and $\psi_i$ are  real numbers in $\left(0,l\right)$, $i=1,2,\cdots,n$, $\sum_{i=1}^{n}\theta_i=ml=\sum_{i=1}^{n}\psi_i$, where $m$ is a positive constant less then $n$. Then

\begin{theo}\label{qe}
Assume that $f\left(\theta\right)$  is a positive $C^2$-function on $\left(0,l\right)$, and it  satisfies
\begin{equation}\label{fc51}
{f^\prime}^2\left(\theta\right)-f\left(\theta\right)f^{\prime\prime}\left(\theta\right)=\mu,\ \ \theta \in \left(0,l\right),
\end{equation}
where $\mu$ is a constant and $f^\prime\left(\theta\right)f^{\prime\prime}\left(\theta\right)\neq0$.
\begin{enumerate}[(i)]
  \item\begin{enumerate}[(1)]
          \item If $f^{\prime\prime}\left(\theta\right)<0$ and $f^\prime\left(\theta\right)>0$ on $\left(0,l\right)$, then
  \begin{equation}\label{fc52}
  2\sum_{i=1}^{n}{f\left(\theta_i\right)f^\prime\left(\psi_i\right)}\geq\sum_{i=1}^{n}{f\left(\theta_i\right)f^\prime\left(\theta_i\right)}+\sum_{i=1}^{n}{f\left(\psi_i\right)f^\prime\left(\psi_i\right)};
  \end{equation}
  \item If $f^{\prime\prime}\left(\theta\right)<0$ and $f^\prime\left(\theta\right)<0$ on $\left(0,l\right)$, then
  \begin{equation}
  2\sum_{i=1}^{n}{f\left(\theta_i\right)f^\prime\left(\psi_i\right)}\leq\sum_{i=1}^{n}{f\left(\theta_i\right)f^\prime\left(\theta_i\right)}+\sum_{i=1}^{n}{f\left(\psi_i\right)f^\prime\left(\psi_i\right)};
  \end{equation}

  \end{enumerate}

  \item\begin{enumerate}[(1)]
         \item If $f^{\prime\prime}\left(\theta\right)>0$ and $f^\prime\left(\theta\right)>0$ on $\left(0,l\right)$, then
  \begin{equation}
  2\sum_{i=1}^{n}{f\left(\theta_i\right)f^\prime\left(\psi_i\right)}\leq\sum_{i=1}^{n}{f\left(\theta_i\right)f^\prime\left(\theta_i\right)}+\sum_{i=1}^{n}{f\left(\psi_i\right)f^\prime\left(\psi_i\right)};
  \end{equation}

\item If $f^{\prime\prime}\left(\theta\right)>0$ and $f^\prime\left(\theta\right)<0$ on $\left(0,l\right)$, then
  \begin{equation}
  2\sum_{i=1}^{n}{f\left(\theta_i\right)f^\prime\left(\psi_i\right)}\geq\sum_{i=1}^{n}{f\left(\theta_i\right)f^\prime\left(\theta_i\right)}+\sum_{i=1}^{n}{f\left(\psi_i\right)f^\prime\left(\psi_i\right)}.
  \end{equation}
       \end{enumerate}
\end{enumerate}
The equalities hold when and only when $\theta_i=\psi_i$.
\end{theo}
\begin{proof}
For the convenience, the following notations will be used in the latter of this proof.

\begin{equation*}
\mathrm{\Theta}=\left(\theta_1,\theta_2,\cdots,\theta_n\right)\in \mathbb{R}^n;\ \ \ \mathrm{\Psi}=\left(\psi_1,\psi_2,\cdots,\psi_n\right)\in \mathbb{R}^n;
\end{equation*}
\begin{equation*}
H_n=\Big\{\mathrm{\Theta}\in \mathbb{R}^n,\sum_{i=1}^{n}\theta_i=ml\Big\}\quad\left(0<m<n\right);
\end{equation*}

\begin{equation*}
I_n=\left(0,l\right)^n\subset \mathbb{R}^n;\ \ \ D_n=H_n\cap I_n.
\end{equation*}

\noindent We set

\begin{equation*}
\begin{aligned}
G\left(\mathrm{\Theta},\mathrm{\Psi}\right)
&=2\sum_{i=1}^{n}{f\left(\theta_i\right)f^\prime\left(\psi_i\right)}-\sum_{i=1}^{n}{f\left(\theta_i\right)f^\prime\left(\theta_i\right)}-\sum_{i=1}^{n}{f\left(\psi_i\right)f^\prime\left(\psi_i\right)}.
\end{aligned}
\end{equation*}

Case (\emph{i}). First, if $f^{\prime\prime}\left(\theta\right)<0$ and $f^\prime\left(\theta\right)>0$ on $\left(0,l\right)$, then $f^\prime\left(\theta\right)$ is a strict decreasing function on $\left(0,l\right)$. Since $G\left(\mathrm{\Psi},\mathrm{\Psi}\right)=0$, we should demonstrate that $G\left(\mathrm{\Theta},\mathrm{\Psi}\right)\geq0$ with equality  only when $\Theta=\Psi$. It suffices for us to verify the following inequality
\begin{equation*}
G\left(\Phi,\Psi\right)>G\left(\Psi,\Psi\right)~for~any~\Phi=\left(\phi_1,\phi_2,\cdots,\phi_n\right)\neq\Psi~in~D_n.
\end{equation*}
In fact the line segment joining $\Phi$ and $\Psi$ can be defined as the following form
\begin{equation*}
\mathrm{\Theta}\left(t\right)=\left(\theta_1\left(t\right),\theta_2\left(t\right),\cdots,\theta_n\left(t\right)\right),
\end{equation*}
in which
\begin{equation*}
\theta_i\left(t\right)=t\psi_i+\left(1-t\right)\phi_i,\ \ \ i=1,2,\cdots,n,\ \ \ t \in \left[0,1\right].
\end{equation*}
Then
\begin{equation*}
\theta_i^\prime\left(t\right)=\psi_i-\phi_i,\ \ \ i=1,2,\cdots,n.
\end{equation*}
Observe that
\begin{equation}\label{fc56}
\sum_{i=1}^{n}\left(\psi_i-\phi_i\right)=ml-ml=0.
\end{equation}
Furthermore, when $\psi_i\neq\phi_i$, we have
\begin{equation}\label{in1}
\left(\psi_i-\phi_i\right)\left[f^\prime\left(\psi_i\right)-f^\prime\left(\theta_i\left(t\right)\right)\right]<0,\ \ \ i=1,2,\cdots,n,\ \ \ t \in \left[0,1\right].
\end{equation}
Notice that $f^\prime\left(\theta\right)$ is decreasing on $\left(0,l\right)$.
If $\psi_i>\phi_i$, then
\begin{equation*}
\theta_i\left(t\right)=t\psi_i+\left(1-t\right)\phi_i<t\psi_i+\left(1-t\right)\psi_i=\psi_i,
\end{equation*}
and $f^\prime\left(\theta_i\left(t\right)\right)>f^\prime\left(\psi_i\right)$.
 If $\psi_i<\phi_i$, then
\begin{equation*}
\theta_i\left(t\right)=t\psi_i+\left(1-t\right)\phi_i>t\psi_i+\left(1-t\right)\psi_i=\psi_i,
\end{equation*}
and $f^\prime\left(\theta_i\left(t\right)\right)<f^\prime\left(\psi_i\right)$.

 Hence, the  inequality (\ref{fc58}) also holds  for $0\leq t \leq1$

\begin{equation}\label{fc58}
\begin{aligned}
&\sum_{i=1}^{n}\left(\psi_i-\phi_i\right)\left[f^\prime\left(\psi_i\right)-f^\prime\left(\theta_i\left(t\right)\right)\right]f^\prime\left(\theta_i\left(t\right)\right)<0.
\end{aligned}
\end{equation}
Besides, from (\ref{fc51}) and (\ref{fc56}) we  have
\begin{equation*}
\begin{aligned}
\mu\sum_{i=1}^{n}\left(\psi_i-\phi_i\right)&=\sum_{i=1}^{n}[{f^\prime}^2\left(\theta_i\left(t\right)\right)-f\left(\theta_i\left(t\right)\right)f^{\prime\prime}\left(\theta_i\left(t\right)\right)]\left(\psi_i-\phi_i\right)
=0.
\end{aligned}
\end{equation*}
Thus
\begin{equation}\label{fc59}
\begin{aligned}
\sum_{i=1}^{n}{\left(\psi_i-\phi_i\right)f\left(\theta_i\left(t\right)\right)f^{\prime\prime}\left(\theta_i\left(t\right)\right)}&=\sum_{i=1}^{n}{\left(\psi_i-\phi_i\right)\left[f^\prime\left(\theta_i\left(t\right)\right)\right]^2}.
\end{aligned}
\end{equation}
Now, differential $G\left(\mathrm{\Theta}\left(t\right),\mathrm{\Psi}\right)$ with respect to $t$ and by  (\ref{fc59}), we have
\begin{equation}
\begin{aligned}
G^\prime\left(\mathrm{\Theta}\left(t\right),\mathrm{\Psi}\right)&=2\sum_{i=1}^{n}{f^\prime\left(\psi_i\right)f^\prime\left(\theta_i\left(t\right)\right)\left(\psi_i-\phi_i\right)}-\sum_{i=1}^{n}\left[2{f^\prime}^2\left(\theta_i\left(t\right)\right)-1\right]\left(\psi_i-\phi_i\right)
\\&=2\sum_{i=1}^{n}{f^\prime\left(\psi_i\right)f^\prime\left(\theta_i\left(t\right)\right)\left(\psi_i-\phi_i\right)}-2\sum_{i=1}^{n}\left.{f^\prime}^2\left(\theta_i\left(t\right)\right)\right.\left(\psi_i-\phi_i\right)
\\&=2\sum_{i=1}^{n}\left(\psi_i-\phi_i\right)\left[f^\prime\left(\psi_i\right)-f^\prime\left(\theta_i\left(t\right)\right)\right]f^\prime\left(\theta_i\left(t\right)\right)<0.
\end{aligned}
\end{equation}
This shows that $G\left(\mathrm{\Theta}\left(t\right),\mathrm{\Psi}\right)$ is decreasing on $\left[0,1\right]$ about the variable $t$, thus
\begin{equation*}
G\left(\mathrm{\Theta}\left(0\right),\mathrm{\Psi}\right)=G\left(\mathrm{\Phi},\mathrm{\Psi}\right)>G\left(\mathrm{\Psi},\mathrm{\Psi}\right)=G\left(\mathrm{\Theta}\left(1\right),\mathrm{\Psi}\right).
\end{equation*}

\noindent Since $\Phi$ is arbitrarily chosen,  the proof of case (i)(1) is complete.

Second, if $f^\prime\left(\theta\right)<0$ then the inequality (\ref{fc58}) is reversed and $G^\prime\left(\mathrm{\Theta}\left(t\right),\mathrm{\Psi}\right)>0$. This shows that $G\left(\mathrm{\Theta}\left(t\right),\mathrm{\Psi}\right)$ is increasing on $\left[0,1\right]$ about the variable $t$, thus
\begin{equation*}
G\left(\mathrm{\Theta}\left(0\right),\mathrm{\Psi}\right)=G\left(\mathrm{\Phi},\mathrm{\Psi}\right)<G\left(\mathrm{\Psi},\mathrm{\Psi}\right)=G\left(\mathrm{\Theta}\left(1\right),\mathrm{\Psi}\right).
\end{equation*}
\noindent Then we complete the proof of case (i)(2).

Case (ii). First, if $f^{\prime\prime}\left(\theta\right)>0$ and $f^\prime\left(\theta\right)>0$ on $\left(0,l\right)$. The proof is similar to case (i)(1), then
\begin{equation*}
\sum_{i=1}^{n}\left(\psi_i-\phi_i\right)\left[f^\prime\left(\psi_i\right)-f^\prime\left(\theta_i\left(t\right)\right)\right]f^\prime\left(\theta_i\left(t\right)\right)>0.
\end{equation*}
This shows that $G\left(\mathrm{\Theta}\left(t\right),\mathrm{\Psi}\right)$ is increasing on $\left[0,1\right]$ about the variable $t$, thus
\begin{equation*}
G\left(\mathrm{\Theta}\left(0\right),\mathrm{\Psi}\right)=G\left(\mathrm{\Phi},\mathrm{\Psi}\right)<G\left(\mathrm{\Psi},\mathrm{\Psi}\right)=G\left(\mathrm{\Theta}\left(1\right),\mathrm{\Psi}\right).
\end{equation*}

Second, if $f^{\prime\prime}\left(\theta\right)>0$ and $f^\prime\left(\theta\right)<0$ on $\left(0,l\right)$, then
\begin{equation*}
\sum_{i=1}^{n}\left(\psi_i-\phi_i\right)\left[f^\prime\left(\psi_i\right)-f^\prime\left(\theta_i\left(t\right)\right)\right]f^\prime\left(\theta_i\left(t\right)\right)<0.
\end{equation*}
This shows that $G\left(\mathrm{\Theta}\left(t\right),\mathrm{\Psi}\right)$ is decreasing on $\left[0,1\right]$ about the variable $t$, thus
\begin{equation*}
G\left(\mathrm{\Theta}\left(0\right),\mathrm{\Psi}\right)=G\left(\mathrm{\Phi},\mathrm{\Psi}\right)>G\left(\mathrm{\Psi},\mathrm{\Psi}\right)=G\left(\mathrm{\Theta}\left(1\right),\mathrm{\Psi}\right).
\end{equation*}
\end{proof}

\begin{theo}\emph{(\cite{Zhang X1997Bonnesen-style})}\label{qa}
Assume that ${\Lambda}_n$ is an  $n$-sided  planar convex polygon inscribed in a circle of radius $R$. Denote by  $L_n$ and  $A_n$  the perimeter and area  of
${\Lambda}_n$, then

\begin{equation}\label{fc511}
A_n-L_nR\cos{\sigma}+d_n\left(R\cos{\sigma}\right)^2\le0,{\ \ \ \ \ d}_n=n\tan{\frac{\pi}{n}}.
\end{equation}
The equality holds when and only when $\theta_1=\cdots=\theta_n=\sigma$.
\end{theo}

\begin{proof}
 Let $f\left(\theta_i\right)=\sin{\theta_i}$, $\theta_i\in(0,\frac{\pi}{2})$ and $f\left(\psi_i\right)=\sin{\sigma}$ where $\psi_i=\sigma=\frac{\pi}{n}$, $i=1,2,\cdots,n (n\ge3)$. Then the following inequality is the special case of (\ref{fc52}), that is

\begin{equation}\label{fc512}
2\cos{\sigma\sum_{i=1}^{n}{\sin\theta_i}}\geq\sum_{i=1}^{n}{\sin\theta_i\cos{\theta_i}}+d_n\left(\cos{\sigma}\right)^2.
\end{equation}

Denote by $a_i$ the length of the $i$th side of ${\Lambda}_n$ and $\theta_i$ the half of the central angle subtended by the $i$th side of ${\Lambda}_n$, $i=1,2,\cdots,n$, then
\begin{figure}[h]
	\centering
	\includegraphics[width=4.6cm]{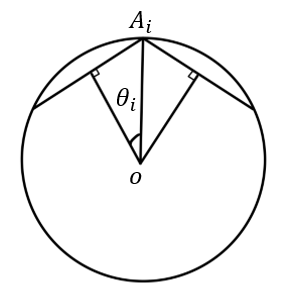}
	\label{2.png}
\end{figure}
\begin{equation}\label{fc513}
L_n=\sum_{i=1}^{n}a_i=\sum_{i=1}^{n}{2R\sin{\theta_i}};\ \ \ A_n=\sum_{i=1}^{n}{\frac{1}{2}a_iR\cos{\theta_i}}=\sum_{i=1}^{n}R^2\sin{\theta_i}\cos{\theta_i};
\end{equation}
\begin{equation}\label{fc514}
L_n^\ast=2Rn\sin{\frac{\pi}{n}};\ \ \ A_n^\ast=R^2n\sin{\sigma\cos{\sigma}}.
\end{equation}
Substituting (\ref{fc513}) and (\ref{fc514}) into (\ref{fc512}), then (\ref{fc511}) is obtained.
\end{proof}

\begin{theo}\emph{(\cite{Ma L2015A Bonnesen-style})}
Assume that ${\Lambda}_n$ is the planar convex $n$-sided polygon with perimeter  $L_n$ and area $A_n$. And ${\Lambda}_n$ is inscribed in a circle with $R$ as its radius. Meanwhile, there exist a regular polygon with perimeter $L_n^\ast$ and area $A_n^\ast$ inscribed in the same circle with ${\Lambda}_n$, then
\begin{equation}\label{fc515}
L_n^2-4d_nA_n\geq\frac{1}{R^2}\left(A_n^\ast-A_n\right)^2,\ \ \ \ d_n=n\tan{\frac{\pi}{n}}.
\end{equation}
The equality holds when and only when ${\Lambda}_n$ is regular polygon.
\end{theo}

\begin{proof}
From Theorem \ref{qa}, we have
\begin{equation}\label{fc516}
\begin{aligned}
L_n^2 &\geq \left[\frac{A_n}{R\cos{\sigma}}+d_n\left(R\cos{\sigma}\right)\right]^2
\\&=\frac{A_n^2}{\left(R\cos{\sigma}\right)^2}+2d_nA_n+d_n^2\left(R\cos{\sigma}\right)^2.
\end{aligned}
\end{equation}
(\ref{fc516}) is equivalent to
\begin{equation*}
\begin{aligned}
L_n^2-4d_nA_n &\geq \frac{A_n^2}{\left(R\cos{\sigma}\right)^2}-2d_nA_n+d_n^2\left(R\cos{\sigma}\right)^2
\\&=\left[\frac{A_n}{R\cos{\sigma}}-d_n\left(R\cos{\sigma}\right)\right]^2
\\&=\frac{1}{\left(R\cos{\sigma}\right)^2}\left(A_n-R^2n\sin{\sigma \cos{\sigma}}\right)^2
\\&=\frac{1}{\left(R\cos{\sigma}\right)^2}\left(A_n^\ast-A_n\right)^2.
\end{aligned}
\end{equation*}
Notice that $A_n^\ast=R^2n\sin{\sigma\cos{\sigma}}$, $d_n=n\tan{\sigma}$, and $\sigma=\frac{\pi}{n}$. It follows that

\begin{equation*}
\begin{aligned}
L_n^2-4d_nA_n &\geq \frac{\cos^2{\sigma}+\sin^2{\sigma}}{\left(R\cos{\sigma}\right)^2}\left(A_n^\ast-A_n\right)^2
\\&=\frac{1}{R^2}\left(1+{tan}^2{\sigma}\right)\left(A_n^\ast-A_n\right)^2
\\&\geq\frac{1}{R^2}\left(A_n^\ast-A_n\right)^2.
\end{aligned}
\end{equation*}
Therefore
\begin{equation*}
L_n^2-4d_nA_n\geq\frac{1}{R^2}\left(A_n^\ast-A_n\right)^2.
\end{equation*}
\end{proof}

\end{document}